\documentclass[12pt,a4paper,openany,oneside]{amsart}
\usepackage[a4paper,left=3cm,right=3cm, top=3cm, bottom=3cm]{geometry}
\usepackage[latin1]{inputenc}
\usepackage{mathrsfs}
\usepackage[T1]{fontenc}
\usepackage{graphicx}
\usepackage{amsmath}
\usepackage{amsthm}
\usepackage{amssymb}
\usepackage{hyperref}
\usepackage{color}
\usepackage{enumerate}
\usepackage{esint}
\theoremstyle{plain}
\newtheorem{thm}{Theorem}[section]
\newtheorem{cor}[thm]{Corollary}
\newtheorem{lem}[thm]{Lemma}
\newtheorem{prop}[thm]{Proposition}
\theoremstyle{definition}
\newtheorem{defn}{Definition}[section]
\theoremstyle{remark}
\newtheorem{remark}{Remark}

\newcommand{\R}{\mathbb{{R}}}
\newcommand{\N}{\mathbb{{N}}}

\newcommand{\Omegabar}{\overline{\Omega}}

\newcommand{\dOmega}{{\partial\Omega}}

\newcommand{\dB}{{\partial B}}

\newcommand{\ddiv}{\text{div}}

\newcommand{\intOmega}{\int_\Omega}
\newcommand{\xki}{x_{k,i}}

\newcommand{\muki}{\mu_{k,i}}

\title[]{A-priori bounds for a quasilinear problem in critical dimension}

\author[G.~Romani]{Giulio Romani}

\address[G.~Romani]{Aix Marseille Universit\'{e}, CNRS, Centrale Marseille, I2M, 39 rue F. Joliot Curie, 13453 Marseille, France}

\email{giulio.romani@univ-amu.fr}

\begin{document}
	
\begin{abstract}
	\noindent We establish uniform a-priori bounds for solutions of the quasilinear problem
	\begin{equation*}
		\begin{cases}
			-\Delta_Nu=f(u)\quad&\mbox{in }\Omega,\\%
			u=0\quad&\mbox{on }\dOmega,
		\end{cases}
	\end{equation*}
	where $\Omega\subset\R^N$ is a bounded smooth convex domain and $f$ is positive, superlinear and subcritical in the sense of the Trudinger-Moser inequality. The typical growth of $f$ is thus exponential. Finally, a generalisation of the result for nonhomogeneous nonlinearities is given. Using a blow-up approach, this paper completes the results in \cite{DP,LRU}, enlarging the class of nonlinearities for which the uniform a-priori bound applies.
\end{abstract}

\maketitle

	\section{Introduction and main results}
	The study of a-priori bounds for solutions of elliptic boundary value problems, that is, establishing  the existence of a positive constant $C$ such that $\|u\|_{L^\infty(\Omega)}\leq C$ for all solutions $u$, is an interesting and important issue. Indeed, on the one hand it is a key point to show existence of solutions by means of the degree theory. Moreover, a large number of different techniques have been developed to get such results for problems of the kind
	\begin{equation}\label{DIRf_2nd_order}
	\begin{cases}
	-\Delta_m u=f(x,u)\quad&\mbox{in }\Omega,\\%
	u=0\quad&\mbox{on }\dOmega,
	\end{cases}
	\end{equation}
	where $\Omega\subset\R^N$ is typically a smooth bounded domain, $\Delta_mu:=\ddiv(|\nabla u|^{m-2}\nabla u)$ and $f:\Omega\times\R\to\R^+$ has a subcritical growth in the second variable. Here subcriticality is meant in the sense of the Sobolev embeddings.\\
	\indent When $m=2$, namely for second-order semilinear problems, uniform a-priori bounds were firstly obtained for a (non-optimal) class of subcritical nonlinearities by Brezis and Turner in \cite{BT} by means of Hardy-Sobolev inequalities. Some years later, Gidas and Spruck in \cite{GS} and de Figueiredo, Lions and Nussbaum in \cite{dFLN} improved this result, applying respectively a blow-up method and a moving-planes technique together with a Poho\v{z}aev identity. The main assumption therein was that the growth at $\infty$ of $f$ is controlled by a suitable power with exponent $2^*-1=\tfrac{N+2}{N-2}$, $2^*$ being the critical Sobolev exponent. We also point out the more recent work of Castro and Pardo \cite{CP} which enlarges the class of nonlinearities involved.\\
	\indent Regarding \textit{quasilinear} equations, similar results have been achieved for $1<m<N$ by Azizieh and Cl\'{e}ment \cite{AC}, Ruiz \cite{Ruiz}, Dall'Aglio et al. \cite{DaDcGP}, Lorca and Ubilla \cite{LU} and Zou \cite{Zou}. In these works, the nonlinearity may depend also on $x$ and on the gradient, nonetheless its growth at infinity with respect to $u$ should be less than a subcritical power. Originally, most of these results concerned the case $1<m\leq2$ due to some symmetry and monotonicity arguments for solutions to $m$-laplacian equations which were at that time available only for that range of $m$; however, those techniques have been later extended also to the case $m>2$ in the papers \cite{DS1,DS2}. See also the recent work of Damascelli and Pardo \cite{DP}, which extends further the class of nonlinearities for which the a-priori bound applies, in the spirit of \cite{dFLN,CP}.
	\vskip0.2truecm
	\indent When, on the other hand, we consider the limiting case $m=N$, the Sobolev space $W^{1,N}_0(\Omega)$ compactly embeds in all Lebesgue spaces and the well-known Trudinger-Moser inequality shows that the maximal growth for $g$ such that $\intOmega g(u)dx<+\infty$ for any $u\in W^{1,N}_0(\Omega)$ is $g(t)=\exp\{\alpha_N t^{\frac{N}{N-1}}\}$, the constant $\alpha_N>0$ being explicit. Therefore, we can consider problems of kind \eqref{DIRf_2nd_order} with suitable exponential nonlinearities. Nevertheless, a-priori bounds may be found only up to the threshold $t\mapsto e^t$, as the work \cite{BM} of Brezis and Merle shows for $m=N=2$ (see also Section \ref{SectionCounterex}). In addition, the authors proved local a-priori bounds when the nonlinearity is of kind $f(x,t)=V(x)e^t$. This, together with the boundary analysis \cite{dFLN}, yields the desired a-priori bound for convex domains. The results in \cite{BM} have been later on extended in the quasilinear setting by Aguilar Crespo and Peral Alonso in \cite{AP} and we refer also to \cite{EM} for concentration compactness issues in this direction.\\ 
	\indent The boundary value problem \eqref{DIRf_2nd_order} with $m=N$ and general subcritical nonlinearities (in the sense of the Trudinger-Moser inequality) has been studied by Lorca, Ruf and Ubilla in \cite{LRU}, where the authors considered superlinear growths either controlled by the map $t\mapsto e^{t^\alpha}$ for some $\alpha\in(0,1)$, or which are comparable to $t\mapsto e^t$. To prove a uniform a-priori bound in the first alternative, techniques involving Orlicz spaces are applied, while for the second the authors adapt the strategy in \cite{BM}.\\
	\indent The goal of our work is to fill the gap between these growths. Indeed, Passalacqua improved the results of \cite{LRU} in his Ph.D. Thesis \cite{Pass} by means of a subtle modification of the argument therein, but the gap was still not completely filled and this seems to be out of reach with those techniques. Here, instead, we apply a blow-up method inspired by \cite{RobWei} to deduce the desired estimates in the interior of the domain. Moreover, our results complete the analysis in \cite{DP} for the case $m=N$ and, further, to the best of our knowledge, they seem to be new even for the semilinear problem when $N=2$.\\
	\indent We note that a similar approach has been also applied by Mancini and the author in \cite{MIO_2} to study higher-order problems.
	\vskip0.4truecm
	
	\noindent Throughout all the paper, $\Omega$ is a bounded and strictly convex domain in $\R^N$ with $C^{1,\alpha}$ boundary and the nonlinearity $f$ satisfies the following conditions:
	\begin{enumerate}[\text{A}1)]
		\item $f\geq0$, $f\in C^1(\R)$ and $f'\geq0$ on $[M, +\infty)$ for some $M\geq 0$;
		\item there exists a positive constant $d$ such that  $\displaystyle\liminf_{t\to+\infty}\frac{f(t)}{t^{N-1+d}}=+\infty$;
		\item there exists $\displaystyle\lim_{t\to+\infty}{\frac{f'(t)}{f(t)}}:=\beta\in[0,+\infty)$.
	\end{enumerate}
	
	The second assumption is standard (cf. \cite{LRU,DP}) and in the linear case $N=2$ is equivalent to say that $f$ is slightly more than superlinear. Assumption (\textit{A3}) is a control on the growth at $\infty$ of $f$. Indeed, Gronwall Lemma implies that for any $\varepsilon>0$ there are constants $C_\varepsilon,D_\varepsilon>0$ such that
	\begin{equation}\label{fexpbeta}
	\min\{0,\,D_\varepsilon e^{(\beta-\varepsilon)t}-C_\varepsilon\}\leq f(t)\leq D_\varepsilon e^{(\beta+\varepsilon)t}+C_\varepsilon.
	\end{equation}
	We refer to \cite[Lemma 1.1]{MIO_2} for a detailed proof.
	
	In the blow-up analysis, it will be useful to distinguish between the following cases.
	\begin{defn}\label{SubQuasi}
		Let $f$ satisfy (\textit{A1})-(\textit{A3}). We say that $f$ is \textit{subcritical} if
		\begin{equation*}\label{subcritical}
		\lim_{t\to+\infty}{\frac{f'(t)}{f(t)}}=0.
		\end{equation*}
		On the other hand, we say that $h$ is \textit{critical} if
		\begin{equation*}\label{quasicritical}
		\lim_{t\to+\infty}{\frac{f'(t)}{f(t)}}\in(0,+\infty).
		\end{equation*}
	\end{defn}
	\begin{remark}\label{Growths}
		In view of \eqref{fexpbeta}, an example of a subcritical nonlinearity $f(t)=\log^{\tau}(1+t)t^p e^{t^\alpha}$ with $\tau\geq 0$, $p\geq1$ and $\alpha\in[0,1)$, and for a critical nonlinearity is $f(t)=\frac{e^{\gamma t}}{(t+1)^q}$ with $\gamma>0$, $q\in\R$.
	\end{remark}
	\indent Although this distinction will be relevant in the proof, the critical case being more delicate, our main result applies for both classes of nonlinearities. 
	
	\begin{thm}\label{thm}
		Let $\Omega\subset\R^N$, $N\geq 2$, be a bounded strictly convex domain of class $C^{1,\gamma}$ for some $\gamma\in(0,1)$ and let $f$ satisfy (\textit{A1})-(\textit{A3}). Then, there exists a constant $C>0$ such that all weak solutions $u\in W^{1,N}_0(\Omega)$ of the problem
		\begin{equation}\label{DIRf}
		\begin{cases}
		-\Delta_Nu=f(u)\quad&\mbox{in }\Omega,\\%
		u=0\quad&\mbox{on }\dOmega,
		\end{cases}
		\end{equation}
		satisfy $\|u\|_{L^\infty(\Omega)}\leq C$.
	\end{thm}
	We recall that $u\in W^{1,N}_0(\Omega)$ is a \textit{weak solution} of problem \eqref{DIRf} if
	$$\intOmega|\nabla u|^{N-2}\nabla u\nabla\varphi=\intOmega f(u)\varphi,\qquad\mbox{for all}\,\,\varphi\in C^\infty_0(\Omegabar).$$
	Notice that, unlike \cite{LRU,DP}, we are not prescribing that our solutions are in $L^\infty(\Omega)$. Indeed, by \cite[Proposition 6.2.2]{Pass}, every $W^{1,N}_0(\Omega)$ weak solution of problem \eqref{DIRf} is classical, namely it belongs to $C^{1,\alpha}(\Omegabar)$ for some $\alpha\in(0,1)$ (see also \cite[Proposition 3.1]{DaDcGP}).
	\begin{remark}
		By the regularity theory of quasilinear problems by \cite{Lieb,Tolk} the $C^{1,\gamma}$ assumptions on $\dOmega$ are sufficient to guarantee that the uniform a-priori estimate in Theorem \ref{thm} is actually in $C^{1,\tilde\gamma}(\Omegabar)$ for some $\tilde\gamma\in(0,1)$.
	\end{remark}
	
	\begin{remark}
		We may set our problem in the framework of \textit{entropy solutions} and obtain the same result. Indeed, one may prove that, under conditions (\textit{A1})-(\textit{A3}), such solutions are indeed weak and in turn classical. For further details, we refer to Section \ref{SectionCounterex}, in particular to Remark \ref{RmkCounterex}.
	\end{remark}
	\vskip0.2truecm
	Once the uniform a-priori bound of Theorem \ref{thm} is established, the existence of a weak solution for problem \eqref{DIRf} may be proved by means of the topological degree theory, provided a control of the behaviour of $f$ in 0 is imposed.
	\begin{prop}\label{Existence}
		In addition of the assumptions of Theorem \ref{thm}, suppose that
		\begin{equation*}
		\limsup_{t\to0^+}\frac{f(t)}{t^{N-1}}<\lambda_1,
		\end{equation*}
		where $\lambda_1$ denotes the first eigenvalue of $-\Delta_N$. Then, problem \eqref{DIRf} admits a positive weak solution.
	\end{prop}
	We will not give the proof of Theorem \ref{Existence}, being rather standard once Theorem \ref{thm} is established, addressing the interested reader to \cite[Theorem 1.5]{DP}.

	\vskip0.2truecm
	This paper is organized as follows: in Section \ref{SectionPrel&Boundary&RHS} we collect some auxiliary results as well as boundary and energy estimates obtained in \cite{LRU}; Section \ref{SectionBlowup} contains the proof of Theorem \ref{thm} distinguishing between the subcritical and the critical case. In Section \ref{SectionGen} we give a generalisation of our result for a class of nonhomogeneous problems and finally in Section \ref{SectionCounterex} we provide an example of nonlinearities which, although being subcritical in the sense of Trudinger-Moser, do not satisfy assumption (\textit{A3}) and for which one can find unbounded solutions.
	
	\section{Preliminary estimates}\label{SectionPrel&Boundary&RHS}
	In this section, we state some known results which will be used in the sequel. We begin with a result by Serrin, together with the regularity theory from \cite{Lieb}. Next, we state a Harnack-type inequality.
	\begin{lem}[\cite{Serrin}, Theorem 2]\label{LocalEstimates_Serrin}
		Let $u$ be a weak solution of $-\Delta_Nu=h$ in $B_{2R}(0)\subset\Omega$. Then,
		\begin{equation}
		\|u\|_{C^{1,\alpha}(B_R(0))}\leq C\big(\|u\|_{L^N(B_{2R}(0))}+KR\big),
		\end{equation}
		for a suitable $\alpha\in(0,1)$ and positive constants $C(\alpha, R, \|h\|_{L^\frac{N}{N-\varepsilon}(\Omega)})$ and $K(R, \|h\|_{L^\frac{N}{N-\varepsilon}(\Omega)})$, for some $\varepsilon\in(0,1)$.
	\end{lem}
	\begin{lem}[\cite{Serrin}, Theorem 6]\label{Harnack_Serrin}
		Let $u\geq0$ be a solution of $-\Delta_Nu=h$ in $B_{3R}(0)\subset\Omega$. Then
		$$\max_{B_R(0)}(u)\leq C\big(\min_{B_R(0)}(u)+K\big),$$
		where the constants $C,K$ depending on $R$ and on $\|h\|_{L^\frac{N}{N-\varepsilon}(\Omega)}$, $\varepsilon\in(0,1)$.
	\end{lem}
	The next result is taken from \cite[Propositions 2.1\,-\,3.1]{LRU} and concerns the behaviour near the boundary and the energy for all solutions of problem \eqref{DIRf}.
	\begin{prop}\label{Prop_boundary&energy}
		There exist positive constants $r$, $C_0$ such that every solution $u\in W^{1,N}_0(\Omega)$ of \eqref{DIRf} verifies
		\begin{equation}\label{bdry}
		u(x)\leq C_0\quad\mbox{and}\quad|\nabla u(x)|\leq C_0,\quad x\in\Omega_r,
		\end{equation}
		where $\Omega_r:=\{x\in\Omega\,|\,d(x,\dOmega)\leq r\}$, and
		\begin{equation}\label{RHS}
		\int_\Omega f(u)\leq\Lambda.
		\end{equation}
	\end{prop}
	Roughly speaking, the proof of the first inequality in \eqref{bdry} is obtained similarly to the analogous estimate in the second-order case (see \cite{dFLN}) by means of a Picone inequality originally proved in \cite{AH} (see also \cite[Theorem 2.5]{DP}). This implies the second inequality in \eqref{bdry} by the standard regularity estimates. Finally, \eqref{RHS} follows from \eqref{bdry} testing the equation with an appropriate function with vanishing gradient outside $\Omega_r$.\\
	We point out that the uniform estimates \eqref{bdry}-\eqref{RHS} do \textit{not} rely on assumption (\textit{A3}), so they can be deduced for any superlinear growth (in the sense of assumption (\textit{A2})).

	\section{Uniform bounds inside the domain}\label{SectionBlowup}
	
	The argument is based on a blow-up method inspired by \cite{RobWei} and developed in \cite{MIO_2}: supposing the existence of an unbounded sequence of solutions, we define a rescaled sequence which locally converges to a solution of a Liouville's equation on $\R^N$. Then, we have to distinguish two cases according to the growth of $f$. If $f$ is subcritical, the right-hand side of the limit equation is constant and this readily yields the contradiction. In the critical framework, the Liouville's equation is still nonlinear, so a more accurate analysis is needed: we shall see that the energy concentrates around the blow-up points with a threshold which is too large for the energy bound of Proposition \ref{Prop_boundary&energy} to hold.
	
	\vskip0.3truecm
	
	Let $(u_k)_k$ a sequence of solutions of \eqref{DIRf} and $(x_k)_k\subset\Omega$ points such that 
	\begin{equation}\label{M_k}
	u_k(x_k)=\|u_k\|_{L^\infty(\Omega)}=:M_k\nearrow+\infty.
	\end{equation}
	Since $\Omega$ is bounded, then, up to a subsequence, $x_k\to x_\infty\in\Omega$, where $d(x_\infty,\dOmega)>0$ by Proposition \ref{Prop_boundary&energy}. Moreover, we define $v_k:\Omega_k\to\R^-$ as
	\begin{equation}\label{vk}
	v_k(x):=u_k(x_k+\mu_kx)-M_k,
	\end{equation}
	where $\Omega_k:=\frac{\Omega-x_k}{\mu_k}$ and
	\begin{equation}\label{muk}
	\mu_k:=\dfrac1{(f(M_k))^{1/N}}.
	\end{equation}
	Notice that $\mu_k\rightarrow0$ implies $\Omega_k\nearrow\R^N$, since $x_\infty\in\Omega$. Moreover, each $v_k$ satisfies (in a weak sense) the equation
	\begin{equation}\label{DeltaN_1}
	-\Delta_N v_k(x)=\frac{f(u_k(x_k+\mu_kx))}{f(M_k)}\in [0,1],\qquad x\in\Omega_k.
	\end{equation}
	Indeed, for any test function $\varphi\in C^\infty_0(\Omegabar)$ there holds
	\begin{equation*}
	\begin{split}
	\int_{\Omega_k}|\nabla v_k|^{N-2}\nabla v_k\nabla\varphi&=\int_{\Omega_k}\mu_k^{N-1}|\nabla u_k|^{N-2}(x_k+\mu_kx)\nabla u_k(x_k+\mu_kx)\nabla\varphi(x)dx\\
	&
	=\mu_k^{N-1}\intOmega|\nabla u_k|^{N-2}(y)\nabla u_k(y)(\nabla \varphi)\big(\tfrac{y-x_k}{\mu_k}\big)\mu_k\frac{dy}{\mu_k^N}\\
	&=\intOmega f(u_k(y))\varphi\big(\tfrac{y-x_k}{\mu_k}\big)dy\\
	&=\int_{\Omega_k} f(u_k(x_k+\mu_kx))\varphi(x)\mu_k^Ndx\\
	&=\int_{\Omega_k}\frac{f(u_k(x_k+\mu_kx))}{f(M_k)}\varphi(x)dx.
	\end{split}
	\end{equation*}
	Moreover, fixing $R>0$ and looking at the behaviour of the sequence $(v_k)_k$ in $B_R(0)$, applying the Harnack inequality given by Lemma \ref{Harnack_Serrin} to $w_k:=-v_k\geq0$, we find
	\begin{equation}\label{Harnack_applied}
	\|v_k\|_{L^\infty(B_R(0))}=\max_{B_R(0)}|v_k|=\max_{B_R(0)} w_k\leq C(R)(\min_{B_R(0)} w_k+K(R))=\tilde K(R).
	\end{equation}
	Therefore, by the local estimate provided by Lemma \ref{LocalEstimates_Serrin}, we infer that, up to a subsequence, $v_k\to v$ in $C^{1,\alpha}_{loc}(\R^N)$. We claim that $v$ is a weak solutions of the Liouville's equation
	\begin{equation}\label{Eq_R^N}
	-\Delta_N v=e^{\beta v}\quad\,\mbox{in}\;\R^N,
	\end{equation}
	where $\beta$ is defined in (\textit{A3}). Indeed, the equation \eqref{DeltaN_1} can be rewritten as
	\begin{equation*}
	-\Delta_N v_k(x)=\exp\big\{\log\circ f(u_k(x_k+\mu_kx))-\log\circ f(M_k)\big\}
	\end{equation*}
	and thus, by the first-order Taylor expansion of $\log\circ f$ around $M_k$, one finds
	\begin{equation}\label{DeltaN_2}
	-\Delta_N v_k(x)=e^{\frac{f'(z_k(x))}{f(z_k(x))}(u_k(x_k+\mu_kx)-M_k)}=e^{\frac{f'(z_k(x))}{f(z_k(x))}v_k(x)},
	\end{equation}
	where $z_k(x):=M_k+\theta(x)(u_k(x_k+\mu_kx)-M_k)=M_k+\theta(x) v_k(x)$, $\theta(x)\in(0,1)$. Since $v_k\to v$ uniformly on compact sets and $M_k\to+\infty$, then $z_k(x)\to+\infty$ uniformly on compact sets, so \eqref{Eq_R^N} follows by taking the limits as $k\to+\infty$ in \eqref{DeltaN_2}.\\
	\indent Consequently, we split our analysis according to whether $f$ is subcritical or critical in the sense of Definition \ref{SubQuasi}.

	\subsection{The subcritical case ($\boldsymbol{\beta=0}$)}\label{Subcritical}
	In this case the equation \eqref{Eq_R^N} satisfied by the limit function $v$ reduces to
	\begin{equation}\label{Eq_R^N_subcritical}
	-\Delta_N v=1\quad\,\mbox{in}\;\R^N.
	\end{equation}
	Recalling the energy estimate \eqref{RHS} of Proposition \ref{Prop_boundary&energy}, it is sufficient to integrate \eqref{Eq_R^N_subcritical} on $\R^N$ to get the desired contradiction, which proves Theorem \ref{thm} for such nonlinearities:
	\begin{equation}\label{assurdo1}
	\begin{split}
	+\infty=\int_{\R^N} dx&=\int_{\R^N}\lim_{k\to+\infty}e^{\frac{f'(z_k(x))}{f(z_k(x))}v_k(x)}\chi_{\Omega_k}(x)dx\\
	&=\int_{\R^N}\lim_{k\to+\infty}e^{\log(f(u_k(x_k+\mu_kx)))-\log(f(M_k))}\chi_{\Omega_k}(x)dx\\
	&\leq\liminf_{k\to+\infty}\int_{\Omega_k}\dfrac{f(u_k(x_k+\mu_kx))}{f(M_k)}dx\\
	&\stackrel{[y=x_k+\mu_kx]}{=}\liminf_{k\to+\infty}\intOmega f(u_k(y))dy\leq\Lambda.
	\end{split}
	\end{equation}

	\subsection{The critical case ($\boldsymbol{\beta>0}$)}
	With the same steps as in \eqref{assurdo1}, one actually gets
	\begin{equation}\label{energiafinita}
	\int_{\R^N} e^{\beta v}dx\leq\Lambda.
	\end{equation}
	Therefore, we can characterize the limit profile $v$ by standard computations from a Liouville-type result by Esposito \cite[Theorem 1.1]{Esp} as
	\begin{equation}\label{charact}
	v(x)=-\frac N\beta\log\left(1+\left(\frac{\beta^{N-1}}{C_N}|x|^N\right)^{\frac 1{N-1}}\right),
	\end{equation}
	where $C_N:=N\left(\frac{N^2}{N-1}\right)^{N-1}$. Furthermore, we claim that the energy concentrates around the blow-up points sequence $(x_k)_k$, namely there exists a constant $\theta>0$ such that
	\begin{equation}\label{stimaintegrale}
	\lim_{R\rightarrow+\infty}\liminf_{k\rightarrow+\infty}\int_{B_{R\mu_k}(x_k)}f(u_k(y))dy\geq\theta>0,
	\end{equation}
	where $\theta$ is characterized by 
	\begin{equation}\label{theta}
		\theta=\frac{N\omega_N}{\beta^{N-1}}\left(\frac{N^2}{N-1}\right)^{N-1},
	\end{equation}
	with $\omega_N$ denoting the volume of the unit ball in $\R^N$. Indeed,
	\begin{equation*}
	\begin{split}
	0<\theta:&=\int_{\R^N}e^{\beta v}=\lim_{R\rightarrow+\infty}\int_{B_R(0)}\lim_{k\rightarrow+\infty}e^{\big(\frac{f'(z_k(x))}{f(z_k(x))}\big)v_k(x)}dx\\
	&\leq\lim_{R\rightarrow+\infty}\liminf_{k\rightarrow+\infty}\int_{B_R(0)}\dfrac{f(u_k(x_k+\mu_kx)}{f(M_k)}dx\\
	&=\lim_{R\rightarrow+\infty}\liminf_{k\rightarrow+\infty}\int_{B_{R\mu_k}(x_k)}f(u_k(y))dy.
	\end{split}
	\end{equation*} 
	
	We can quantify the constant $\theta$ as follows. Let us define the auxiliary function $w:=\beta v+(N-1)\log\beta$. Then, it is easy to show that $w$ satisfies $-\Delta_Nw=e^w$ and $\int_{\R^N}e^w<\infty$. Therefore, by \cite[Theorem 1.1]{Esp}, we have
	\begin{equation*}
	\theta=\int_{\R^N}e^{\beta v}=\frac1{\beta^{N-1}}\int_{\R^N}e^w=\frac{N\omega_N}{\beta^{N-1}}\left(\frac{N^2}{N-1}\right)^{N-1}.
	\end{equation*}
	\vskip0.2truecm
	In other words, if we apply the blow-up technique around a \textit{maximum point}, we see that we can characterize the limit profile and we get the concentration of the energy as in \eqref{stimaintegrale}-\eqref{theta}. This is indeed what happens around \textit{any blow-up point} of the sequence $(u_k)_k$, meaning points belonging to the set
	$$S:=\{y\in \bar \Omega\,|\,\exists (y_j)_j \subseteq \Omega\,|\, y_j\to y, u_{k_j}(y_j)\to+\infty \text{ as } j\to +\infty\}.$$ 
	\begin{lem}\label{lemmabrutto}
		{From the sequence $(u_k)_k$ one can extract a subsequence still denoted by $(u_k)_k$ for which the following holds.}\\
		There are $P\in\N$ and sequences $(x_{k,i})_k$ for $1\leq i\leq P$ such that $\lim_{k\to+\infty}u_k(x_{k,i})=+\infty$ for any $i$ and, setting
		$$v_{k,i}(x):=u_k(\xki+\muki x)-u_k(\xki),\qquad\muki:=(f(u_k(\xki)))^{-1/N},$$
		and $x^{(i)}=\lim_{k\to+\infty}x_{k,i}$, we have
		\begin{enumerate}[(i)]
			\item $\lim_{k\to+\infty}\frac{|\xki-x_{k,j}|}{\muki}=+\infty$ for $1\leq i\neq j\leq P$;
			\item $v_{k,i}\to v$ in $C^{1,\gamma}_{loc}(\R^N),$ for $1\leq i\leq P$, where $v$ is defined in \eqref{charact} and \eqref{stimaintegrale}-\eqref{theta} still hold;
			\item $\sup_k\inf_{1\leq i\leq P}|x-\xki|^Nf(u_k(x))\leq C$ for every $x\in\Omega$.
		\end{enumerate}
	\end{lem}
	\begin{proof}
		The proof is inspired by \cite[Claims 5-7]{RobWei}. We say that the \textit{property} $\mathcal{H}_p$ holds whenever there exist $p$ sequences $(x_{k,i})_k$ for $1\leq i\leq p$ such that $u_k(x_{k,i})\to+\infty$ as $k\to+\infty$ and such that (\textit{i}) and (\textit{ii}) hold. From our previous investigation, we know that $\mathcal{H}_1$ holds. Our aim is to prove that the number of such sequences is finite by showing the following: if $\mathcal{H}_p$ holds then either $\mathcal{H}_{p+1}$ holds too or there exists a constant $C>0$ such that (\textit{iii}) holds. Indeed, if so, supposing by contradiction that $\mathcal{H}_p$ holds for any $p$, then we would be able to find a sequence of disjoint balls $B_{R\muki}(\xki)$ such that \eqref{stimaintegrale} holds, and thus by Lemma \ref{Prop_boundary&energy},
		\begin{equation*}
		\Lambda\geq\intOmega f(u_k(x))dx\geq\int_{\bigcup_{i=1}^p B_{R\muki}(\xki)} f(u_k(x))dx=\sum_{i=1}^p\int_{B_{R\muki}(\xki)} f(u_k(x))dx\geq p\theta,
		\end{equation*}
		an upper bound for $p$, a contradiction. As a consequence, $\mathcal{H}_p$ does not hold for any $p$ and, by the claimed alternative, the proof is completed.\\
		\indent Now the claim has to be proved. Suppose that $\mathcal{H}_p$ for some $p\in\N$ but not (\textit{iii}). Hence, defining $w_k(x):=\inf_{1\leq i\leq p}|x-\xki|^Nf(u_k(x))$, one can find a sequence of points $(y_k)_k$ such that $w_k(y_k)=\|w_k\|_\infty\nearrow+\infty$. We show that these points $(y_k)_k$, together with $(x_{k,i})_k$ for $1\leq i\leq P$ let $\mathcal{H}_{p+1}$ holds. First $u_k(y_k)\to+\infty$, since
		$$+\infty\leftarrow w_k(y_k)\leq (diam(\Omega))^Nf(u_k(y_k))$$
		and by the local boundedness of $f$ on $[0,+\infty)$. Moreover, defining
		$$\tilde u_k(x):=u_k(y_k+\gamma_k x)-u_k(y_k),\qquad\gamma_k:=(f(u_k(y_k)))^{-1/N},$$
		first we have
		\begin{equation}\label{i_1}
		\inf_{1\leq i\leq p}\frac{|y_k-x_{k,i}|}{\gamma_k}=\inf_{1\leq i\leq p}|y_k-x_{k,i}|f(u_k(y_k))^{1/N}=w_k(y_k)^{1/N}\to+\infty.
		\end{equation}
		Then, suppose by contradiction that $|y_k-x_{k,i}|=O(\muki)$, that is $y_k=\xki+\muki\theta_{k,i}$, with $|\theta_{k,i}|\leq C$. We have
		\begin{equation*}
		\begin{split}
		w_k(y_k)&\leq|y_k-\xki|^Nf(u_k(y_k))=\muki^N|\theta_{k,i}|^Nf(u_k(y_k))=|\theta_{k,i}|^N\frac{f(u_k(y_k))}{f(u_k(\xki))}\\
		&=|\theta_{k,i}|^Ne^{\frac{f'(z_{k,i}(\theta_{k,i}))}{f(z_{k,i}(\theta_{k,i}))}v_{k,i}(\theta_{k,i})}\to|\theta_{\infty,i}|^Ne^{\beta v(\theta_{\infty,i})}=\frac{|\theta_{\infty,i}|^N}{\big(1+\beta {C_N}^\frac1{1-N}|\theta_{\infty,i}|^{\frac N{N-1}}\big)^N}<+\infty,
		\end{split}
		\end{equation*}
		a contradiction, so (\textit{i}) is proved.\\
		\indent Now, it remains to prove that $u_k$ has the same blow-up behaviour also for the sequence $(y_k)_k$. To this aim, we first show that, for any $R>0$
		\begin{equation}\label{ii_1}
		\frac{f(u_k(y_k+\gamma_kx))}{f(u_k(y_k))}\leq C
		\end{equation}
		for some $C=C(R)>0$ and for any $x\in B_R(0)$. Notice that this is not obvious as in \eqref{DeltaN_1}, since $(y_k)_k$ do not have to be necessarily maximum points for $(u_k)_k$. Rewriting in terms of the sequence $(u_k)_k$ the inequality $w_k(y_k+\gamma_kx)\leq w_k(y_k)$, which holds for any $x\in B_R(0)$ for $k$ large enough, one finds
		\begin{equation}\label{ii_2}	
		\frac{f(u_k(y_k+\gamma_kx))}{f(u_k(y_k))}\leq \bigg(\frac{\inf_{1\leq i\leq p}|y_k-\xki|}{\inf_{1\leq i\leq p}|y_k+\gamma_kx-\xki|}\bigg)^N.
		\end{equation}
		Fix now $\varepsilon\in(0,1)$. By \eqref{i_1}, there exists $\bar k=\bar k(R,\varepsilon)>0$ such that for any $k>\bar k$ we have $R\gamma_k\leq\varepsilon
		|y_k-\xki|$. Therefore $|y_k+\gamma_kx-\xki|\geq|y_k-\xki|-\gamma_kR\geq(1-\varepsilon)|y_k-\xki|$ and so \eqref{ii_1} follows from \eqref{ii_2} with $C=(1-\varepsilon)^{-N}$.\\
		In order to complete the proof of the local compactness of the sequence $(\tilde u_k)_k$ and consequently get (\textit{ii}), we need to have its local boundedness in the $L^N$ norm according to Lemma \ref{LocalEstimates_Serrin}. This time, unlike \eqref{Harnack_applied}, before applying the Harnack inequality of Lemma \ref{Harnack_Serrin}, we need to know that our sequence is uniformly bounded from above, which is not immediate. So, suppose by contradiction that for any sequence of positive constants $(C_k)_k\nearrow+\infty$ there exist points $(\tilde x_k)_k\subset B_R(0)$ such that $\tilde u_k(\tilde x_k)\geq C_k$ holds, thus, since $f$ is increasing,
		\begin{equation*}
		f(u_k(y_k+\gamma_k\tilde x_k))\geq f(C_k+u_k(y_k)).
		\end{equation*}
		Together with \eqref{ii_1}, this implies
		\begin{equation}\label{iii_3}
		f(C_k+u_k(y_k))\leq Cf(u_k(y_k)).
		\end{equation}
		Choosing now $C_k=e^{u_k(y_k)^2}$ and setting $t_k:=u_k(y_k)$, \eqref{iii_3} rereads as
		\begin{equation}\label{iii_4}
		f(e^{t_k^2}+t_k)\leq Cf(t_k).
		\end{equation}
		Then, by superlinearity of $f$ and by the fact that $f(s)\leq e^{\gamma s}+D$ for some $\gamma,D>0$ for any $s\in\R^+$ by (\textit{A3}), we have
		\begin{equation*}
		e^{t_k^2}+t_k-\bar C\leq f(e^{t_k^2}+t_k)\leq Cf(t_k)\leq e^{\gamma t_k}+D,
		\end{equation*}
		which yields a contradiction since $t_k\nearrow+\infty$. Hence, $M(R):=\max_{B_R(0)}\tilde u_k\in[0,+\infty)$ and we can consider the function $U_k:=M(R)-\tilde u_k$. We have $U_k\geq 0$, $\min_{B_R(0)}U_k=0$ and furthermore $U_k$ satisfies
		$$-\Delta_NU_k=-\frac{f(u_k(y_k+\gamma_kx))}{f(u_k(y_k))}\in[-C(R),0],$$
		where the bound from below follows from \eqref{ii_1}.
		Therefore, by Lemma \ref{Harnack_Serrin}, we get
		\begin{equation*}
		\big|M(R)+\min_{B_R(0)}\tilde u_k\big|=\max_{B_R(0)}|U_k|=\max_{B_R(0)}U_k\leq C(R)\big[\min_{B_R(0)}U_k+K(R)\big]=C(R)K(R),
		\end{equation*}
		which in turn implies
		\begin{equation}\label{eq_Harnack_tilde}
		\bigg|\max_{B_R(0)}\tilde u_k-\big|\min_{B_R(0)}\tilde u_k\big|\bigg|\leq \tilde C(R).
		\end{equation}
		Therefore, either $\max_{B_R(0)}\tilde u_k\geq\big|\min_{B_R(0)}\tilde u_k\big|$, and in this case we easily infer $\|\tilde u_k\|_{L^\infty(B_R(0))}\leq C(R)$, or we deduce from \eqref{eq_Harnack_tilde} that
		\begin{equation*}
		\big|\min_{B_R(0)}\tilde u_k\big|-\max_{B_R(0)}\tilde u_k\leq \tilde C(R).
		\end{equation*}
		This readily implies $\min_{B_R(0)}\tilde u_k\geq -C(R)$, so again we find $\|\tilde u_k\|_{L^\infty(B_R(0))}\leq C(R)$.\\
		\noindent Applying Lemma \ref{LocalEstimates_Serrin}, we infer that $(\tilde u_k)_k$ is locally compact in $C^{1,\alpha}_{loc}(\R^N)$ and, with similar steps as in \eqref{DeltaN_2}-\eqref{charact}, we finally infer (\textit{ii}). Consequently, the property $\mathcal{H}_p$ is satisfied and the proof is concluded.
	\end{proof}

	\begin{cor}\label{Sfinite}
		The set of blow-up points is finite.
	\end{cor}
	\begin{proof}
		We show the finiteness of $S$ by proving that $S=\{x^{(i)},\,1\leq i\leq P\}$, where the points $x^{(i)}$ are defined as in Lemma \ref{lemmabrutto}. Indeed, let $\bar x\not\in \{x^{(i)},\,1\leq i\leq P\}$ for which one can find a sequence $\bar x_k\to\bar x$ such that, up to a subsequence, $u_k(\bar x_k)\to+\infty$. We have $\inf_{1\leq i\leq P,\, k\in\N}|\bar x_k-x^{(i)}|\geq \bar d>0$ and, by Proposition \ref{Prop_boundary&energy}, $d(\bar x_k,\dOmega)\geq\eta>0$ for some constants $\bar d,\eta$. Then, by \textit{(iii)} of Lemma \ref{lemmabrutto} and the superlinearity of $f$, we infer
		$$C\geq \inf_{1\leq i\leq P}|\bar x_k-\xki|^Nf(u_k(\bar x_k))\geq\frac{\bar d^N}{2}(u_k(\bar x_k)-C),$$
		a contradiction.
	\end{proof}
	
	Now we are in the position to prove Theorem \ref{thm}. In broad terms, we will show that the amount of the energy concentrating at blow-up points found in \eqref{stimaintegrale}-\eqref{theta} is too large to have at the same time the upper-bound of Proposition \ref{Prop_boundary&energy}.
	
	\begin{proof}[Proof of Theorem \ref{thm}]
		By \eqref{M_k} $S$ is nonempty. Let $x_0\in S$ and $r$ small such that $S\cap B_r(x_0)=\{x_0\}$. In this way, $\|u_k\|_{L^\infty(K)}<C$ for any $K\subset B_r(x_0)\setminus\{x_0\}$ and thus $u_k\to \tilde u$ in $C^{1,\alpha}_{loc}(B_r(x_0)\setminus\{x_0\})$ by Lemma \ref{LocalEstimates_Serrin}. Moreover, by the energy bound of Proposition \ref{Prop_boundary&energy}, the sequence $(f_k)_k:=(f(u_k))_k$ is bounded in $L^1$, thus $f_k\rightharpoonup\mu$ in $\mathcal M(B_r(x_0))\cap L^\infty_{loc}(B_r(x_0)\setminus\{x_0\})$ where $\mu$ is a positive measure which is singular at $x_0$. Hence, we can decompose $\mu$ as
		\begin{equation}\label{dec_mu}
		\mu=A(x)dx+a_0\delta_{x_0},
		\end{equation}
		where $0\leq A\in L^1(B_r(x_0))$, $a_0$ is a positive constant and $\delta_{x_0}$ denotes the Dirac distribution centered at $x_0$. We first claim that $a_0\geq\theta$. Indeed, for any $t\in(0,r)$ we have
		\begin{equation}\label{a_0>theta}
		\int_{B_t(x_0)}d\mu=\lim_{k\to+\infty}\int_{B_t(x_0)}f(u_k(x))dx\geq\liminf_{k\to+\infty}\int_{B_{R\mu_k}(x_k)}f(u_k(x))dx\geq\theta
		\end{equation}
		by \eqref{stimaintegrale}, where here $x_k$ is a blow-up sequence converging to $x_0$. As $t$ is arbitrary, we thus find $a_0\geq\theta$.
		
		Let now $w$ be the distributional solution of
		\begin{equation}\label{DIRdelta}
		\begin{cases}
		-\Delta_Nw=a_0\delta_{x_0}\quad&\mbox{in }B_r(x_0),\\%
		w=0\quad&\mbox{on }\partial B_r(x_0).
		\end{cases}
		\end{equation}
		Then by \cite[Theorem 2.1]{KV} $w\in C^{1,\alpha}(B_r(x_0)\setminus\{x_0\})$ and has the explicit form
		\begin{equation*}
		w(x)=\left(\frac{a_0}{N\omega_N}\right)^{\frac1{N-1}}\log\left(\frac r{|x-x_0|}\right).
		\end{equation*}
		We claim that
		\begin{equation}\label{wu}
		w\leq \tilde u\quad\;\mbox{in}\; B_r(x_0).
		\end{equation}
		If so, choosing $\varepsilon\in(0,\frac\beta N)$, in view of \eqref{fexpbeta} on the one hand by Proposition \ref{Prop_boundary&energy} we get
		\begin{equation*}
		\begin{split}
		\int_{B_r(x_0)}e^{(\beta-\varepsilon)w}&\leq \int_{B_r(x_0)}e^{(\beta-\varepsilon)\tilde u}\leq\liminf_{k\to+\infty}\int_{B_r(x_0)}e^{(\beta-\varepsilon)u_k}\\
		&\leq c\liminf_{k\to+\infty}\int_{B_r(x_0)}f(u_k)+d\leq C(\Lambda).
		\end{split}		
		\end{equation*}
		On the other hand, by means of the explicit expressions for $w$ and $\theta$:
		\begin{equation*}
		\begin{split}
		\int_{B_r(x_0)}e^{(\beta-\varepsilon)w}&=\int_{B_r(x_0)}\left(\frac r{|x-x_0|}\right)^{(\beta-\varepsilon)\left(\frac{a_0}{N\omega_N}\right)^{\frac1{N-1}}}dx\\
		&\geq\int_{B_r(x_0)}\left(\frac r{|x-x_0|}\right)^{(\beta-\varepsilon)\left(\frac\theta{N\omega_N}\right)^{\frac1{N-1}}}dx\\
		&\geq\int_{B_r(x_0)}\left(\frac r{|x-x_0|}\right)^{\frac{(\beta-\varepsilon)}{\beta}\left(\frac{N^2}{N-1}\right)}dx\\
		&>\int_{B_r(x_0)}\left(\frac r{|x-x_0|}\right)^Ndx=+\infty,
		\end{split}
		\end{equation*}
		a contradiction. Therefore the proof of Theorem \ref{thm} is completed once we prove \eqref{wu}. Here, we follow the strategy in \cite{LRU}. For any $j\in\N$, let us define the maps $B_j:\overline\R\to\R^+$ by
		\begin{equation*}
		B_j(s)=\begin{cases}
			0\quad&\mbox{for }s<0,\\%
			s\quad&\mbox{for }0\leq s\leq j,\\
			j\quad&\mbox{for }s>j,
		\end{cases}
		\end{equation*}
		and the functions $z^{(j)}_k:=B_j(w-u_k)\geq 0$. Then, it is easy to see that $z^{(j)}_k(x_0)=j$ for any $k$, that  ${z^{(j)}_k}_{|_{\partial B_r(x_0)}}=0$ and, moreover, $z^{(j)}_k\in W^{1,N}_0(B_R(x_0))$. Furthermore, as $B_j$ is continuous, one has $z^{(j)}_k\to z^{(j)}$ pointwise, where 
		\begin{equation*}
		z^{(j)}(s)=\begin{cases}
		B_j(w-\tilde u)\quad&\mbox{for }x\not=x_0,\\
		j\quad&\mbox{for }x=x_0.
		\end{cases}
		\end{equation*}
		Notice that $z^{(j)}_k$ may be extended to 0 in $\Omega\setminus B_r(x_0)$. Since $w$ and $u_k$ solve respectively problems \eqref{DIRdelta} and \eqref{DIRf}, then
		\begin{equation*}
		\begin{split}
		\int_{B_r(x_0)}\left(|\nabla w|^{N-2}\nabla w - |\nabla\tilde u|^{N-2}\nabla\tilde u\right)\nabla z^{(j)}_k&=a_0 z^{(j)}_k(x_0)-\intOmega f(u_k)z^{(j)}_k\\
		&=a_0 j-\int_{B_r(x_0)} f(u_k)z^{(j)}_k.
		\end{split}
		\end{equation*}
		Recalling $f(u_k)\rightharpoonup \mu$ and \eqref{dec_mu}, using the generalized Fatou's Lemma with measures (see e.g. \cite[Chapter 11, Proposition 17]{Royden}), one infers
		$$\liminf_{k\to+\infty}\int_{B_r(x_0)}f(u_k)z^{(j)}_kdx\geq\int_{B_r(x_0)}z^{(j)}d\mu\geq\int_{\{x_0\}}z^{(j)}d\mu\geq a_0 j,$$
		and thus
		\begin{equation*}
		\int_{B_r(x_0)}\left(|\nabla w|^{N-2}\nabla w - |\nabla\tilde u|^{N-2}\nabla\tilde u\right)\nabla z^{(j)}\leq 0,
		\end{equation*}
		that is,
		\begin{equation*}
		\int_{B_r(x_0)\cap\{0\leq w-\tilde u\leq j\}}\left(|\nabla w|^{N-2}\nabla w - |\nabla\tilde u|^{N-2}\nabla\tilde u\right)\nabla\left(w-\tilde u\right)\leq 0,
		\end{equation*}
		which holds for any choice of $j\in\N$. Then, the well-known inequality
		$$d_N|X-Y|^N\leq\langle|X|^{N-2}X-|Y|^{N-2}Y,X-Y\rangle,$$
		for any $X\not=Y\in\R^N$, the constant $d_N>0$ depending only on $N$ (see  \cite[Proposition 4.6]{RenWei}), implies
		\begin{equation*}
		d_N\int_{B_r(x_0)}|\nabla (w-\tilde u)^+|^N\leq 0,
		\end{equation*}
		which finally proves our claim $w-\tilde u\leq 0$ in $B_r(x_0)$, since $\left(w-\tilde u\right)^+\leq 0$ on $\partial B_r(x_0)$.
	\end{proof}

	\section{Generalization to nonhomogeneous problems}\label{SectionGen}
	So far, we studied the homogeneous quasilinear problem \eqref{DIRf}, which allows a clearer exposition and a more direct comparison with the results in \cite{LRU,DP}. However, a similar analysis can be carried out also in the \textit{nonhomogeneous} setting, provided some conditions at infinity and near the boundary are fulfilled. More precisely, we may consider the problem
	\begin{equation}\label{DIRh}
	\begin{cases}
	-\Delta_Nu=h(x,u)\quad&\mbox{in }\Omega,\\%
	u=0\quad&\mbox{on }\dOmega,
	\end{cases}
	\end{equation}
	where $h:\Omegabar\times\R\to\R^+\cup\{0\}$ is a Carath\'{e}odory function satisfying the following conditions:
	\begin{enumerate}[H1)]
		\item $h\in L^\infty(\Omega\times[0,\tau])$ for all $\tau\in\R^+$ and $h(x,t)>0$ for any $x\in\Omega$ and $t>0$;
		\item there exist $f\in C^1([0,+\infty))$ satisfying (\textit{A1})-(\textit{A3}) and $0<a\in L^\infty(\Omega)\cap C(\Omega)$ such that $$\lim_{t\rightarrow+\infty}\frac{h(x,t)}{a(x)f(t)}=1\qquad\mbox{uniformly in}\;\Omega\,;$$
		\item there exist $\bar r,\bar\delta>0$ such that
		\begin{itemize}
			\item $h(\cdot, t)\in C^1(\Omega_{\bar r})$ for all $t\geq 0$ and $\frac{\partial h}{\partial t}(x,t)\geq 0$ in $\Omega_{\bar r}\times\R^+$;
			\item $\nabla_xh(x,t)\cdot\theta\leq 0$ for all $x\in\Omega_r$, $t\geq0$ and unit vectors $\theta$ such that $|\theta-n(x)|\leq\bar\delta$.
		\end{itemize}
	\end{enumerate}
	We recall $\Omega_r:=\{x\in\Omega\,|\,d(x,\dOmega)\leq r\}$.
	\begin{remark}
		In (\textit{H2}) we assume $a>0$ only in the interior of $\Omega$, but it may vanish on the boundary. Moreover, assumption (\textit{H3}) is imposed to uniformly control solutions near the boundary by a moving-planes technique, prescribing in broad terms that $h$ should be decreasing in $x$ along suitable outer directions and increasing in $t$ in a suitable neighborhood of $\dOmega$.
	\end{remark}
	\begin{remark}
		From (\textit{H1})-(\textit{H2}) it follows that for each $\varepsilon>0$ there exists a constant $d_\varepsilon\geq0$ such that
		\begin{equation}\label{asymptotic_GROWTH}
		(1-\varepsilon)a(x)f(t)-d_\varepsilon\leq h(x,t)\leq (1+\varepsilon)a(x)f(t)+d_\varepsilon\quad\;\mbox{for all \,$t\geq0$, $x\in\Omega$}.
		\end{equation}
	\end{remark}
	\begin{thm}\label{thmX}
		Let $\Omega\subset\R^N$, $N\geq 2$, be a bounded strictly convex domain of class $C^{1,\gamma}$ for some $\gamma\in(0,1)$ and let $h$ satisfy (\textit{H1})-(\textit{H3}).
		Then, there exists a constant $C>0$ such that
		\begin{equation*}
		\|u\|_{L^\infty(\Omega)}\leq C
		\end{equation*}
		for all weak solutions $u\in W^{1,N}_0(\Omega)$ of problem \eqref{DIRf}.
	\end{thm}
	\vskip0.2truecm
	The proof of Theorem \ref{thmX} mainly follows the arguments of the previous sections, so we sketch here only the remarkable modifications.
	\begin{proof}
		First, we want to prevent blow-up near the boundary, so we adapt in our context the results of Proposition \ref{Prop_boundary&energy}. With the same argument as in \cite[Proposition 2.1]{LRU} and taking \eqref{asymptotic_GROWTH} into account, one infers that solutions are bounded in $L^d_{loc}(\Omega)$, namely
		\begin{equation}\label{L1_loc}
		\intOmega u^d\Phi_1^N\leq C,
		\end{equation}
		where $d$ is defined in \textit{(A2)}, $C$ is independent of $u$ and $\Phi_1\in W^{1,N}_0(\Omega)$ is the first (positive) eigenfunction of $-\Delta_N$ on $\Omega$. In order to prove that \eqref{L1_loc} yields a uniform bound near $\dOmega$, we have to show that here our solutions are there decreasing with respect to outer directions. We apply a moving-planes technique in the spirit of \cite[Lemma 3.2]{dFdOR}. Let us first fix some notation. For any direction $\nu$, set
		\begin{equation*}
		T_\lambda^\nu:=\{x\in\R^N\,|\,x\cdot\nu=\lambda\}\;\quad\mbox{and}\;\quad\Omega_\lambda^\nu:=\{x\in\Omega\,|\,x\cdot\nu<\lambda\},
		\end{equation*}
		the latter being nonempty for $\lambda>a(\nu):=\inf_{x\in\Omega}x\cdot\nu$. Moreover, for any $x\in\Omega$, we denote by $x_\lambda^\nu$ the symmetric point with respect to the hyperplane $T_\lambda^\nu$, namely $$x_\lambda^\nu=R_\lambda^\nu(x):=x+2(\lambda-x\cdot\nu)\nu.$$ Similarly, we define $u_\lambda^\nu:=R_\lambda^\nu(u)$, which is well-defined on $(\Omega_\lambda^\nu)':=R_\lambda^\nu(\Omega_\lambda^\nu)$. Our aim is to compare $u$ and $u_\lambda^\nu$ near the boundary in the case $\nu=n(x)$ (the outer normal) and $\lambda-a(\nu)$ small, in order to conclude that $u$ is decreasing along these directions in a small neighborhood of $\dOmega$.
		First notice that by convexity of $\Omega$, as long as $\lambda-a(\nu)$ is small, we have $(\Omega_\lambda^\nu)'\subset\Omega$ and $\Omega_\lambda^\nu\cup (\Omega_\lambda^\nu)'\subset\Omega_{\bar r}$. Therefore, for such $\lambda$, in $\Omega_\lambda^\nu$ one has
		\begin{equation*}
		\begin{split}
		-\Delta_N u_\lambda^\nu (x)-\big(-\Delta_N u(x)\big)&=h(x_\lambda^\nu,u_\lambda^\nu)-h(x,u)\stackrel{(H3)}{\geq} h(x,u_\lambda^\nu)-h(x,u)\\
		&=M(x)\big(u_\lambda^\nu-u\big),
		\end{split}
		\end{equation*}
		where
		$$M(x)=\frac{\partial h}{\partial t}(x,\eta(x,\lambda))\stackrel{(H3)}{\geq} 0$$
		for some real number $\eta(x,\lambda)$ lying between $u_\lambda^\nu(x)$ and $u(x)$ by the mean value theorem, hence $M\in L^\infty(\Omega_\lambda^\nu)$. Therefore, $u_\lambda^\nu$ et $u$ satisfy
		\begin{equation*}
		\begin{cases}
		-\Delta_Nu_\lambda^\nu-M(x)u_\lambda^\nu\geq -\Delta_Nu-M(x)u\quad&\mbox{in }\Omega_\lambda^\nu,\\%
		u_\lambda^\nu\geq u\quad&\mbox{on }\dOmega_\lambda^\nu,
		\end{cases}
		\end{equation*}
		and the weak comparison principle in small domains \cite[Theorem 1.3]{DS1} yields $u_\lambda^\nu\geq u$ in $\Omega_\lambda^\nu$. We point out that the mentioned result is originally stated for $M$ positive constant and for homogeneous nonlinearities, but it can be easily adapted to our setting (here assumption (\textit{H1}) is required). By arbitrariness of $\lambda$ and $x\in\Omega_{\bar r}$, one deduces that $u$ is decreasing along the outer normal direction and then, by a compactness argument, that there exist $\delta\in(0,\bar \delta]$ and $r\in(0,\bar r]$ independent of $u$ such that $\nabla u(x)\cdot\theta\leq 0$ for every $x\in\Omega_r$ and for every direction $\theta$ such that $|\theta-n(x)|<\delta$. The boundedness of $u$ and $\nabla u$ in a neighborhood of the boundary now follows by a standard argument by \cite{dFLN}, and moreover one infers
		\begin{equation*}
		\int_\Omega h(x,u)dx\leq\Lambda,
		\end{equation*}
		with $\Lambda$ independent of $u$. We refer to \cite[Propositions 2.1-3.1]{LRU} for the details.
		\vskip0.2truecm
		\indent Let us now address to the blow-up analysis. We shall see that the argument carried out in Section \ref{SectionBlowup} applies to the problem \eqref{DIRh} with only minor modifications. In particular, the rescaled functions $v_k$ defined in \eqref{vk}-\eqref{muk} turn out to be weak solutions of
		\begin{equation*}
		-\Delta_N v_k(x)=\frac{h(x_k+\mu_kx,u_k(x_k+\mu_kx))}{f(M_k)},\qquad x\in\Omega_k:=\frac{\Omega-x_k}{\mu_k},
		\end{equation*}
		and one shows that $v_k\to v$ in $C^{1,\alpha}_{loc}(\R^N)$ which satisfies
		\begin{equation*}
		-\Delta_N v=a(x_\infty)e^{\beta v}\quad\,\mbox{in}\;\R^N.
		\end{equation*}
		We recall that $d(x_\infty,\dOmega)>0$, as we have already excluded boundary blow-up, thus $a(x_\infty)>0$ by (\textit{H2}).\\
		\indent In the subcritical case, the same argument as in Section \ref{Subcritical} holds and, in the critical framework, Lemma \ref{lemmabrutto} easily adapts, showing that near blow-up points the energy concentrates:
		\begin{equation*}
		\lim_{R\rightarrow+\infty}\liminf_{k\rightarrow+\infty}\int_{B_{R\mu_k}(x_k)}h(y,u_k(y))dy\geq\theta>0,
		\end{equation*}
		with the same constant $\theta$ defined in \eqref{theta}. Then, the conclusion of the proof is analogous to the one for the homogeneous case.
	\end{proof}
	
	\section{A counterexample}\label{SectionCounterex}
	As briefly mentioned in the Introduction, our assumption (\textit{A3}) can be interpreted as a control from above for the growth at $\infty$ of the nonlinearity $f$ by a suitable power of $e^t$. In the spirit of Brezis-Merle \cite[Example 2]{BM}, we show an example of nonlinearities which are still subcritical in the sense of the Trudinger-Moser inequality without satisfying (\textit{A3}) and for which one may find a positive potential $a(x)$ such that the problem \eqref{DIRh} admits \textit{unbounded} solutions, in a suitable \textit{distributional} sense.
	\vskip0.2truecm
	To this aim, let $f(t)=e^{t^\alpha}$ with $\alpha\in(1,\frac N{N-1})$. Notice that $f$ clearly satisfies (\textit{A1}) and (\textit{A2}), it is subcritical in the sense of Trudinger-Moser, but the condition (\textit{A3}) is not fulfilled. Moreover, fix parameters $\delta:=\frac{(\alpha-1)N+1}{\alpha N}>0$ and $\gamma:=\frac{\alpha-1}{\alpha}\in(0,1)$ and consider the family of functions $$\phi_\beta(t)= t+\beta t^{\gamma}- \delta\log t,$$
	where $\beta>0$ is a positive parameter that will be chosen later. Notice that $\phi_\beta(t)>t-\delta\log t$ and $\phi_\beta'(t)>1-\frac dt$. Hence, defining
	$$u_\beta(r):= \phi_\beta(l(r))^\frac{1}{\alpha}, \quad r = |x|,$$
	where $l(r):=\log\big(\frac1r\big)$, there exists $\rho_0>0$ such that for any $r\in(0,\rho_0)$ there holds $u_\beta(r)>0$ and $u_\beta'(r)<0$. In the sequel, we adapt the strategy in \cite[\S  6]{MIO_2}.\\
	First, using an implicit function argument as in \cite[Lemma 6.2]{MIO_2}, one can show (up to a smaller value of $\rho_0$) that for any $\rho\in(0,\rho_0)$ there exists $\beta=\beta(\rho)\sim l(\rho)^{\frac1\alpha}$ as $\rho\to0$ such that $u_{\beta(\rho)}(\rho)-\frac{\beta(\rho)}\alpha=0$ on $\dB_{\rho}(0)$. With this choice, one has $u_{\beta(\rho)}(r)\sim l(r)^{\frac1\alpha}$ as $r\to 0$. Then, for $r\in(0,\rho)$ we compute
	\begin{equation}\label{N-lapl_counterex}
	\begin{split}
	-\Delta_Nu_{\beta(\rho)}(r)&=-\frac d{dr}\left[r^{N-1}|u_{\beta(\rho)}'(r)|^{N-2}u_{\beta(\rho)}'(r)\right]\\
	&=\frac{N-1}{\alpha^Nr^N}\left[\bigg(1-\frac1\alpha\bigg)u_{\beta(\rho)}^{N-1-\alpha N}(r)\phi_{\beta(\rho)}'(l(r))^N-u_{\beta(\rho)}^{(1-\alpha)(N-1)}(r)\phi_{\beta(\rho)}'(l(r))^{N-2}\phi_{\beta(\rho)}''(l(r))\right].
	\end{split}
	\end{equation}
	Using the relations  $\frac12\leq\phi_{\beta(\rho)}'\leq 1$ and $\phi_{\beta(\rho)}''\leq\delta l(r)^{-2}$ for $r$ sufficiently small, one infers from \eqref{N-lapl_counterex} that there exists $\rho_1<\rho_0$ such that for any $\rho\in(0,\rho_1)$ and  $r\in(0,\rho)$ one has $-\Delta_Nu_{\beta(\rho)}(r)>0$.
	
	\noindent Let us now fix $\rho\in(0,\rho_1)$ and define
	\begin{equation}\label{u_counterex}
	w(r):=N^{\frac1\alpha}\bigg(u_{\beta(\rho)}(r)-\frac{\beta(\rho)}\alpha\bigg).
	\end{equation}
	We have that $w>0$ solves pointwise
	\begin{equation*}
	\begin{cases}
	-\Delta_Nw=a(x)e^{w^\alpha}\quad&\mbox{in }B_\rho(0)\setminus\{0\},\\%
	w=0\quad&\mbox{on }\partial B_\rho(0),
	\end{cases}
	\end{equation*}
	with $a(x):=-e^{-w^\alpha}\Delta_Nw>0$. Now we claim that $a\in L^\infty(B_\rho(0))$ but $w\not\in L^\infty(B_\rho(0))$. Indeed,
	\[
	\begin{split}
	w^\alpha (r)&= Nu_{\beta(\rho)}^\alpha(r)\bigg(1- \frac{\beta(\rho)}{\alpha u_{\beta(\rho)}(r)}\bigg)^\alpha=Nu_{\beta(\rho)}^\alpha(r)-N\beta(\rho)u_{\beta(\rho)}^{\alpha-1}(r)+o(1)\\
	&=Nl(r)+N\beta(\rho)l(r)^\gamma-N\delta\log l(r)-N\beta(\rho) l(r)^{\frac{\alpha-1}{\alpha}}+o(1)
	\end{split}
	\] 
	as $r\to 0$. Recalling now that $\gamma = \frac{\alpha-1}{\alpha}$, we get 
	\begin{equation*}\label{w_in0}
	w^{\alpha}(r) = Nl(r)-N\delta\log l(r)+o(1)
	\end{equation*}
	as $r\to 0$. Therefore,
	\begin{equation*}
	\begin{split}
	a(x)=-N^\frac1\alpha\Delta_Nu_{\beta(\rho)}\,e^{-w^{\alpha}}& = \frac{N^{\frac{1}{\alpha}}(N-1)(\alpha-1)}{\alpha^{N+1}r^N} l(r)^{\frac{N-1-\alpha N}{\alpha}} r^N l(r)^{N\delta} (1+o(1))\\
	& = \frac{N^{\frac{1}{\alpha}}(N-1)(\alpha-1)}{\alpha^{N+1}} +o(1).
	\end{split}
	\end{equation*}
	as $r\to0$, by our choice of $\delta$. Hence $a\in L^\infty(B_\rho(0))$. However, $w$ is not bounded near 0 as it inherits the same behaviour of $u_{\beta(\rho)}$.
	
	\vskip0.2truecm
	Now we want to prove that $w$ is an \textit{entropy solution} of the problem \eqref{DIRh}. This kind of solutions has been introduced in the context of quasilinear elliptic problems in \cite{entropy_sol} in order to weaken the notion of weak solution for problems with $L^1$-data. Here we recall the precise meaning.
	\begin{defn}
		We say that $u\in T^{1,N}_0(\Omega)$ if $u$ is measurable and $T_ku\in W^{1,N}_0(\Omega)$ for any $k>0$, where $T_ku$ is defined as the truncation of $u$ at level $k$, namely
		\begin{equation*}
		T_k(s)=\begin{cases}
		-k\quad&\mbox{for }s<-k,\\%
		s\quad&\mbox{for }-k\leq s\leq k,\\
		k\quad&\mbox{for }s>k.
		\end{cases}
		\end{equation*}
	\end{defn}
	\begin{defn}\label{entropysol:Def}
		A function $u\in T^{1,N}_0(\Omega)$ is called an \textit{entropy solution} of problem \eqref{DIRh} if for any test function $\varphi\in C^\infty_0(\Omega)$ there holds
		\begin{equation}\label{entropysol:def}
		\int_{\{|u-\varphi|<k\}}|\nabla u|^{N-2}\nabla u(\nabla u-\nabla\varphi)dx\leq\int_{\Omega}T_k(u-\varphi)h(x,u(x))dx.
		\end{equation}
	\end{defn}
	\noindent We start proving $T_kw\in W^{1,N}_0(B_\rho(0))$ for any $k>0$ and $w$ defined in \eqref{u_counterex}. Indeed, $T_kw$ is radial and
	\begin{equation*}
	T_kw(r)=\begin{cases}
	w(r)\quad&\mbox{for }\,r\in [r_k,\rho],\\
	w(r_k)\quad&\mbox{for }\,r\in [0,r_k),
	\end{cases}
	\end{equation*}
	for some suitable $r_k\in(0,\rho)$, so it is easy to see that
	\begin{equation*}
	\int_{r_k}^\rho\bigg|\frac{dw}{dr}\bigg|r^{N-1}dr\sim\int_{r_k}^\rho w(r)^{N(1-\alpha)}\frac{(N-\gamma(1-\log r))^N}{r}dr<C_k.
	\end{equation*}
	In order to show that $u$ satisfies \eqref{entropysol:def}, let $\varphi\in C^\infty_0(B_\rho(0))$ and notice that $T_k(w-\varphi)\in W^{1,N}_0(B_\rho(0))$. By our choice of $a$, the problem \eqref{DIRh} is pointwise satisfied, thus we may test it in $B_\rho(0)$ by $T_k(w-\varphi)$:
	\begin{equation*}
	-\int_{B_\rho(0)}\Delta_N wT_k(w-\varphi)dx=\int_{B_\rho(0)} a(x)e^{w^\alpha} T_k(w-\varphi)dx.
	\end{equation*}
	Note that all is well-defined as we proved $a\in L^\infty(B_\rho(0))$ and $e^{w^\alpha}\in L^1(B_\rho(0))$. Integrating by parts the left-hand side and recalling $T_k(w-\varphi)_{|_{\partial B_\rho(0)}}=0$, we have
	\begin{equation*}
	\int_{B_\rho(0)}|\nabla w|^{N-2}\nabla w \nabla T_k(w-\varphi)dx=\int_{B_\rho(0)} a(x)e^{w^\alpha} T_k(w-\varphi)dx
	\end{equation*}
	which is the case of the equality in \eqref{entropysol:def}.
	\vskip0.2truecm
	We have therefore showed that, despite $w$ is an entropy solution of the problem
	\begin{equation*}
	\begin{cases}
	-\Delta_Nw=a(x)e^{w^\alpha}\quad&\mbox{in }B_\rho(0),\\%
	w=0\quad&\mbox{on }\partial B_\rho(0),
	\end{cases}
	\end{equation*}
	however $w\notin L^\infty(B_\rho(0))$, as $w(r)\sim l(r)^\frac1\alpha$ as $r\to0$.

\begin{remark}\label{RmkCounterex}
	This counterexample shows that the class of the nonlinearities considered by assumption (\textit{A3}) is sharp in order to have the property of (uniform) boundedness within the class of entropy solutions. Indeed, coming back to problem \eqref{DIRh} under assumption (\textit{A3}) and looking for entropy solutions in the sense of Definition \ref{entropysol:Def}, since this time (\textit{A3}) implies a control of the nonlinearity by a suitable power of $e^t$, say $p$, one has
	$$\intOmega|h(x,u(x))|^\alpha dx\leq C\intOmega e^{\alpha pu}<+\infty,$$
	where the last inequality is due to \cite[Corollary 1.7]{AP}. Therefore entropy solutions are classical, and our analysis and thus Theorem \ref{thm} apply.
\end{remark}

\end{document}